\documentclass[11pt]{article}

\usepackage{amsmath, amssymb, amsthm, verbatim,enumerate,bbm,color,mathtools} 

\usepackage{mathrsfs,enumitem}

\usepackage[backref,colorlinks,citecolor=blue,bookmarks=true]{hyperref}

\DeclareMathOperator{\PL}{PL}

\title{An Elementary Analysis of the Prime Partition Function}

\author{Asaf Cohen Antonir \thanks{School of Mathematical Sciences, Tel Aviv University, Tel Aviv, 6997801, Israel.
Email: asafc1$@$tauex.tau.ac.il} \and Asaf Shapira \thanks{
School of Mathematics, Tel Aviv University, Tel Aviv 69978, Israel.
Email: asafico$@$tau.ac.il. Supported in part by ISF Grant 1028/16 and ERC Starting Grant 633509.}}

\date{\today}
\allowdisplaybreaks

\theoremstyle{plain}
\newtheorem{theorem}{Theorem}[section]
\newtheorem{lemma}[theorem]{Lemma}

\makeatletter
\def\moverlay{\mathpalette\mov@rlay}
\def\mov@rlay#1#2{\leavevmode\vtop{%
   \baselineskip\z@skip \lineskiplimit-\maxdimen
   \ialign{\hfil$\m@th#1##$\hfil\cr#2\crcr}}}
\newcommand{\charfusion}[3][\mathord]{
    #1{\ifx#1\mathop\vphantom{#2}\fi
        \mathpalette\mov@rlay{#2\cr#3}
      }
    \ifx#1\mathop\expandafter\displaylimits\fi}
\makeatother

\makeatletter
\renewenvironment{proof}[1][\proofname]
{\par\pushQED{\qed}
	\normalfont\topsep6\p@\@plus6\p@\relax\trivlist
	\item[\hskip\labelsep\bfseries#1\@addpunct{.}]
	\ignorespaces}
{\popQED\endtrivlist\@endpefalse}
\makeatother

\RequirePackage{color}\definecolor{RED}{rgb}{1,0,0}\definecolor{BLUE}{rgb}{0,0,1} 

\begin{document}
\date{}
\maketitle

\begin{abstract}
Let $pp (n)$ denote the
number of ways to write $n$ as a sum of primes. In this paper, we show that
$$ 
\log pp (n) \sim 2\pi\sqrt{\frac{n}{3\log n}}\;.
$$
While sharper estimates are already known, they rely on highly involved and lengthy proofs. In sharp contrast, our approach uses a short, elementary recipe that easily adapts to yield similar asymptotic estimates for several related, extensively studied problems.
\end{abstract}

\section{Introduction}\label{sec:intro}

Let $p(n)$ denote the partition function of $n$, namely the number of sequences of positive integers $a_1\leq a_2\leq \ldots\leq a_k$ such that $\sum_{i=1}^{k}a_i=n$.
More generally, given $A$ a set of positive integers, we let $p_A(n)$ be the \emph{$A$-restricted partition function}, which counts the number of sequences $a_1\leq a_2\leq \ldots \leq a_k$ such that $\sum_{i=1}^{k}a_i=n$ and also for all $i$ we have $a_i\in A$.

The study of various properties
of restricted partition functions is amongst the oldest topics in Combinatorics and Number theory. Some classical examples are Euler's Pentagonal Numbers Theorem and the Rogers--Ramanujan identities (for a more thorough background, see \cite{Andrews1976,Andrews1976-1}).

A very natural property is that of estimating the asymptotic behaviour of $A$-restricted partition function, for different sets $A$.
Arguably, the most well-known result of this type is the celebrated Hardy--Ramanujan \cite{HR2} formula (discovered independently by Uspensky \cite{Usp1920}). It states that
\begin{equation}\label{equation - Hardy Ramanujan}
    p(n) \sim \frac{1}{4n\sqrt{3}} \exp(\pi\sqrt{2n/3})\;.
\end{equation}
Following \cite{HR2}, asymptotic estimates for $p_A(n)$ were obtained for various sets $A$. Examples of this arise already in the work of Hardy and Ramanujan \cite{HR1}, where they obtained bounds analogous to \eqref{equation - Hardy Ramanujan} when $A$ is the set of primes, when $A$ is the set of odd integers, and when $A$ is the set of $d^{th}$ powers of positive integers. These restricted partition functions were further studied by many researchers; see e.g.\ \cite{Gaf2016, Gaf2021, YaLWu2021, YaLChen2016, YaLChen2018}.

Many of the proofs of such asymptotic estimations use highly non-trivial arguments, which are lengthy and complex. In contrast to these proofs, using a short and elementary argument Erd\H{o}s \cite{Erdos} proved only a slightly weaker version of (\ref{equation - Hardy Ramanujan}), that is
\begin{equation}\label{eqErdos}
\log p(n) \sim \pi\sqrt{2n/3}\;.
\end{equation}
In this note, we use the ideas from the elegant inductive argument of Erd\H{o}s, in order to prove estimates similar to \eqref{eqErdos} for various restricted partition functions as well as estimates for the number of `plane partitions'.
Although our proofs do not yield the state of the art estimations, the proofs are significantly shorter and yield the correct asymptotic behaviour of the logarithm of the functions under study.
In this note, we will determine the asymptotic behaviour of the logarithm of the number of plane partitions, partitions into $d^{th}$ powers, and partitions into primes.

Despite the clear differences in each of the above-mentioned counting functions, our proofs regarding their asymptotic behaviour follow a similar three-step process:
\begin{enumerate}[label=\textbf{(R\arabic*)}]
    \item\label{R1} Establish a recursive formula for the relevant partition function, akin to the one used by Erd\H{o}s.
    \item\label{R2} Use the recursive formula to show that the partition function in hand is bounded by a sum of the form $\sum_{k=1}^{\infty}\sum_{a\in A} \varphi(a) e^{-\psi(a,k)}$, where $\varphi(x),\psi(x,y)$ are some functions (which differ in each case) and $A$ is some set of integers.
    \item\label{R3} Evaluate these sums --- in an elementary way --- and conclude the proof.
\end{enumerate}
As will become evident below, all proofs are elementary and closely follow the spirit of Erd\H{o}s's proof of (\ref{eqErdos}) as described in \cite{Nathanson}.

\paragraph{Partitions into Powers:} We generalize restricted partition function a bit further to be able to discuss the partition function restricted to `$d^{th}$ powers' for $d>0$. Indeed, for $\sigma\colon \mathbb{N}_+\to \mathbb{N}_+$ a non-decreasing sequence, we denote the `$\sigma$-restricted' partition function of $n$ by $p_{\sigma}(n)$ to be the number of sequences $a_1\leq a_2\leq \ldots \leq a_k$ such that $\sum_{i=1}^{k}\sigma(a_i)=n$.
This is clearly a generalization of $A$-restricted partition functions for infinite $A$, as we can take $\sigma$ to enumerate $A$ in increasing order.
For every positive $d>0$ and a nonnegative integer $n$ we denote $p_d(n)$ to be $p_{\sigma}(n)$, where $\sigma(x)=\lfloor x^{d} \rfloor$.
Note that in the case of integer $d$, the function $p_d(n)$ is $p_A(n)$ where $A=\{x^{d}:x\in \mathbb{N}_+\}$. The asymptotic estimation we prove for $p_d(n)$ is the following:
\begin{theorem}\label{thm:powers}
For every $d>0$ we have the following when $n\to \infty$:
$$
    \log p_{d}(n) \sim c_d \cdot n^{1/(d+1)}\quad\text{where}\quad  c_d=(d+1)\left(\frac{1}{d}\cdot \Gamma\left(1+\frac{1}{d}\right)\cdot \zeta\left(1+\frac{1}{d}\right)\right)^{\frac{d}{d+1}}\;.
$$
\end{theorem}
In the above theorem, as in the rest of the paper, $\Gamma(z)$ is the {\em Gamma function} given by $\Gamma(z) = \int_0^\infty t^{z-1} e^{-t} \,dt$ and $\zeta(s)$ is the {\em Euler–Riemann zeta function} given by $\zeta(s) = \sum_{n=1}^{\infty} \frac{1}{n^s}$.

We should point out that a more precise asymptotic formula is known for $p_d(n)$. The case of integers $d$ is due to \cite{Gaf2016, HR1}, and the case of non-integer $d$ is more recent and is due to \cite{LucRal2016, YaLWu2021, YaLChen2016, YaLChen2018}. However, these estimates rely
on far longer and highly technical proofs.

The proof of Theorem \ref{thm:powers} is simpler when assuming $d$ is an integer. Therefore, we first prove the theorem for an integer $d$, and then modify the proof to work also for a non-integer $d$. In both cases (integer $d$ or not), we will only prove the upper bound in Theorem \ref{thm:powers}, as the proof of the lower bound is very similar.

\paragraph{Partitions into Primes:} The main result of this paper is an asymptotic estimation of $\log p_{\mathcal{P}}(n)$ where $\mathcal{P}$ is the set of prime numbers. For notional convenience we write $pp(n)$ instead of $p_{\mathcal{P}}(n)$. Then the asymptotic formula of $\log pp(n)$ is given by the following theorem.
\begin{theorem}\label{thm:primes}
We have the following when $n\to \infty$:
$$
        \log pp(n) \sim 2\pi\sqrt{\frac{n}{3\log n}}\;.
$$
\end{theorem}

This theorem was first proven by Hardy-Ramanujan \cite{HR1}, and a more precise result was shown in \cite{Vau2008}.
The proof of Theorem \ref{thm:primes} is the longest among the proofs presented in this note, while remaining elementary. The proofs of the upper and lower bounds in Theorem \ref{thm:primes} are, as before, very similar. However, in contrast to the case of Theorem \ref{thm:powers}, we provide a fully detailed proof for both bounds to illustrate how one can use our recipe to produce lower bounds.

\paragraph{Plane Partitions:}
A very well-studied generalization of the partition function are the \emph{high dimensional partitions}. One special case of interest is $\PL(n)$, the two-dimensional partition function, also known as the plane partition function.
It is defined as the number of arrays of nonnegative integers $(\pi_{i,j})$ such that $\sum_{i,j}\pi_{i,j}=n$ and such that the rows and columns are non-increasing; that is, for all $k$, the sequences $\pi_{i,k}$ and $\pi_{k,i}$ are non-increasing in $i$.

For more general partition functions, such as the plane partition function, one can still apply our three-step recipe, while some of the steps may no longer be elementary. We exemplify this with the plane partition function $\PL(n)$, and show the following:
\begin{theorem}\label{thm:plane_partitions}
We have the following when $n\to \infty$:
$$
\log \PL(n) \sim \left(27\cdot \zeta(3)/4\right)^{1/3} \cdot n^{2/3}\;.
$$
\end{theorem}
We remark that our proof uses Lemma \ref{lem:recursion_plane_partitions}, which establishes \ref{R1} for $\PL(n)$. The lemma provides the required recursion for $\PL(n)$ and is due to MacMahon \cite{Mac1,Mac2,Mac3} (see also \cite{HeiNeuTro2023}). Unfortunately, the proof of the lemma is not as simple as the proofs of the recursion that we use for the restricted partition functions; in particular, it uses the generating function of $\PL(n)$, which is derived in a non-elementary way. For this reason, we do not provide a proof for this lemma, and only show how to implement \ref{R2} and \ref{R3} in an elementary way given this lemma to obtain the upper bound in Theorem \ref{thm:plane_partitions}.

We also wish to highlight that a more precise asymptotic formula is known \cite{Wri1931}, although the proof of this formula uses the generating function of the plane partition as well as analytic techniques; for more on this, see \cite{Sta1971}. The current state of the art regarding the asymptotic growth of higher dimensional partitions is only of the order of magnitude of the logarithm of the number of high dimensional partitions \cite{Solids}.
The reason for the discrepancy between the one and two dimensional partitions and the higher ones comes from the lack of a `nice' formula for its generating function. For a more thorough discussion on this topic, see \cite{Andrews1976}.

\paragraph{Organization of the paper:} The first two sections are `warmup' sections where only upper bounds are proven due to the conceptual similarity of the lower bound cases. In Section \ref{sec:planepartition}, we prove the upper bound of Theorem \ref{thm:plane_partitions}, which illustrates how to implement \ref{R2} and \ref{R3}. Then, in Section \ref{sec:proofpowers}, we prove the upper bound of Theorem \ref{thm:powers} for integer $d$, which will require implementing \ref{R1}. Section \ref{sec:proofpowers2} builds upon Section \ref{sec:proofpowers} and completes the proof of the upper bound of Theorem \ref{thm:powers} for all $d>0$. Finally, in Section \ref{sec:mainproof} and Section \ref{sec:proofprimelower}, we prove the upper bound and the lower bound of Theorem \ref{thm:primes}, respectively. The appendix is dedicated to proving some elementary inequalities.

\section{The Number of Plane Partitions}\label{sec:planepartition}

As mentioned in the introduction, step \ref{R1} requires establishing a recursion formula. In the case of plane partitions, we use the following lemma by MacMahon \cite{Mac1,Mac2,Mac3}:
\begin{lemma}\label{lem:recursion_plane_partitions}
    For every positive integer $n$ we have
    \[
         n\cdot \PL(n)= \sum_{t=1}^{n} \sum_{k=1}^{n/t}t^2 \cdot \PL(n-tk)\;.
    \]
\end{lemma}

We continue by implementing the rest of the recipe and proving Theorem \ref{thm:plane_partitions}.

\begin{proof}[Proof of Theorem \ref{thm:plane_partitions}]
  First, we note that for every $\delta>0$
\begin{equation}\label{eq:exp_sum_powers_0}
    \sum_{m=1}^{\infty} m^2 e^{-\delta m} = \frac{e^{-\delta}(e^{-\delta}+1)}{(1-e^{-\delta})^3}\leq \frac{2}{\delta^3}\;,
\end{equation}
where the equality holds as for all $0<t<1$ we have\footnote{This follows by differentiating the identity $\sum_{k=0}^{\infty}t^{k}=\frac{1}{1-t}$ twice.}
$\sum_{m=1}^{\infty}m^2t^m=\frac{t(t+1)}{(1-t)^3}$, and the inequality holds as for all $x>0$ we have\footnote{See Lemma \ref{lem:estimations} in the appendix for a short proof.} $\frac{e^{-x}(e^{-x}+1)}{(1-e^{-x})^3}\leq \frac{2}{x^3}$.

Now letting $c= (27\zeta(3)/4)^{1/3}$ we prove that for all $n$ we have $\PL(n)\leq\exp(c\cdot n^{2/3})$, which yields the upper bound in Theorem \ref{thm:plane_partitions}. We do so by induction on $n$, with the base case trivially holding. Observe the following:
\begin{align*}
    n\cdot \PL(n) &= \sum _{t=1}^{n} \sum_{k=1}^{n/t} t^2 \cdot \PL(n-tk)\leq \sum _{t=1}^{n} \sum_{k=1}^{n/t} t^2\cdot e^{c(n-tk)^{2/3}} \leq e^{cn^{2/3}}\cdot \sum_{t=1}^{n} \sum _{k=1}^ {n/t} t^2 \cdot e^{-\frac{2ctk}{3n^{1/3}}}\;,
\end{align*}
where the first equality is Lemma \ref{lem:recursion_plane_partitions}, the first inequality is induction, and the second inequality uses the elementary fact
$(n-x)^{2/3}\leq n^{2/3}-\frac{2x}{3n^{1/3}}$.

To finish the proof, we need to show that the double sum above is at most $n$.
This follows by letting $\delta(k)=\frac{2ck}{3n^{1/3}}$ in \eqref{eq:exp_sum_powers_0}, as then we obtain
\begin{equation*}
\pushQED{\qed}
    \sum_{t=1}^{n}\sum_{k=1}^{n/t} t^2 e^{-\delta(k) t} \leq \sum_{k=1}^{\infty}\frac{2}{\delta(k)^3}=\frac{27n}{4c^3}\sum_{k=1}^{\infty}\frac{1}{k^{3}}=n\;.\qedhere
\end{equation*}
\end{proof}

We emphasize that in the above proof, \eqref{eq:exp_sum_powers_0} can be seen as step \ref{R3}, while \ref{R2} is implemented with $A$ being the set of all integers and both $\varphi(x)=x^2$ and $\psi(x,y)=x\delta(y)$.

\section{Partitions into integer Powers}\label{sec:proofpowers}

We begin with the proof of step \ref{R1}, which is used in the proofs of Theorem \ref{thm:powers} and Theorem \ref{thm:primes} as the first step of our recipe.
The following recursion is well known, and was also used implicitly in Erd\H{o}s' proof \cite{Erdos}.

\begin{lemma}\label{lem:recursion}
For every $\sigma\colon \mathbb{N}_+\to \mathbb{N}_+$, a non-decreasing sequence, and positive integer $n$ we have
\begin{equation*}
    n\cdot p_{\sigma}(n)= \sum_{s\in \mathbb{N}} \sigma(s)\sum _{k=1}^{n/\sigma(s)} p_{\sigma}(n-\sigma(s)k)\;.
\end{equation*}
\end{lemma}

\begin{proof}
To see this let $p_{\sigma}(n,s,t)$ and $p_{\sigma}'(n,s,t)$ to be the number of sequences $a_1\leq a_2\leq \ldots \leq a_k$ such that $\sum_{i=1}^{k}\sigma(a_i)=n$ where $s$ appears exactly $t$ times in $\{a_1,a_2,\ldots ,a_k\}$ and at least $t$ times in $\{a_1,a_2,\ldots ,a_k\}$ respectively.
Then by double counting we have
\begin{align*}
    n\cdot p_\sigma(n)&=\sum_{s\in \mathbb{N}, t\in \mathbb {N}} \sigma(s)\cdot t \cdot p_\sigma(n,s,t)=\sum _{s\in \mathbb{N}}\sigma(s)\sum_{t\in \mathbb {N}}t\cdot p_\sigma(n,s,t)\\
    &=\sum _{s\in \mathbb{N}}\sigma(s)\sum_{t\in \mathbb {N}}p_\sigma'(n,s,t)=\sum _{s\in \mathbb{N}}\sigma(s)\sum _{k=1}^{n/\sigma(s)}p_\sigma(n-\sigma(s)k)\;.\qedhere
\end{align*}
\end{proof}

We now prove the upper bound in Theorem \ref{thm:powers} for any \emph{integer} $d>0$.
We begin by proving inequality \eqref{eq:exp_sum_powers} which is an analogue to inequality (\ref{eq:exp_sum_powers_0}) which we used in the previous section. Let $f_\delta(x)=x^{d}e^{-\delta x^{d}}$ and note that for fixed $\delta$ we have $f'_\delta(x)=d(1-\delta x^d)x^{d-1}e^{-\delta x^d}$. Hence, $a\coloneqq \frac{1}{\delta^{1/d}}$ is the maximal point of $f_\delta(x)$. In particular, the function $f_\delta(x)$ is increasing for all $0\leq x\leq a$, and decreasing otherwise.
Therefore,
\begin{equation}\label{eq:floors_powers}
     \sum_{m=1}^{\lfloor a\rfloor-1} f_\delta(m) \leq \int_{0}^{\lfloor a\rfloor}f_\delta(x) dx \quad\text{and}\quad \sum_{m=\lfloor a \rfloor+2}^{\infty} f_\delta(m) \leq \int_{\lfloor a\rfloor}^{\infty}f_\delta(x) dx\;.
\end{equation}
We see that the integral $\int_{0}^{\infty}f_\delta(x)dx$ accounts for $\sum_{m=1}^{\infty}f_\delta(m)$ except possibly for $f(\lfloor a\rfloor)$ and $f(\lfloor a\rfloor +1)$, hence,
\begin{equation}\label{eq:exp_sum_powers}
    \sum_{m=1}^{\infty}f_\delta(m)\leq \int_{0}^{\infty}f_\delta(x)dx+2f(a)=\frac{1}{d\delta^{1+{1}/{d}}}\int_{0}^{\infty}z^{1/d}e^{-z} dz+\frac{2}{e\delta}=\frac{\Gamma\left(1+{1}/{d}\right)}{d\delta^{1+1/d}}+\frac{2}{e\delta}\;,
\end{equation}
where the inequality holds by \eqref{eq:floors_powers} and as $a$ is the maximal point of $f_\delta(x)$, and the first equality holds by the substitution $z=\delta x^d$ which implies that $\frac{z^{1/d-1}}{d\delta^{1/d}}dz=dx$ and by noting that $x=\left(\frac{z}{\delta}\right)^{1/d}$.

Now using \eqref{eq:exp_sum_powers}, for every $\varepsilon>0$, we prove by induction on $n$ that $p_d(n) \leq K e^{(1+\varepsilon)c_dn^{1/(d+1)}}$, where $K=K(\varepsilon)$. Note that by choosing $K$ large enough, we can assume that the induction assumption holds for all $n \leq n_0(\varepsilon,d)$, allowing us to assume in what follows that $n \geq n_0(\varepsilon,d)$.
By Lemma \ref{lem:recursion}, the induction hypothesis, and by setting $\delta(k)\coloneqq \frac{(1+\varepsilon)c_d k}{(d+1)n^{d/(d+1)}}$ we have:
\begin{align}\label{eq:1.0}
    n\cdot p_{d}(n)&= \sum_{m=1}^{n^{1/d}} m^{d}\sum _{k=1}^{n/m^d} p_{d}(n-m^dk) \leq \sum_{m=1}^{n^{1/d}}m^d \sum _{k=1}^{n/m^{d}} K e^{(1+\varepsilon)c_d(n-m^dk)^{1/(d+1)}}\nonumber\\
    &\leq Ke^{(1+\varepsilon)c_d n^{1/(d+1)}}\sum _{k=1}^{n} \sum_{m=1}^{\infty}f_{\delta(k)}(m)\;,
\end{align}
where in the second inequality we used the basic fact $(n-x)^{1/(d+1)} \leq n^{1/(d+1)}-\frac{x}{(d+1)n^{d/(d+1)}}$.

To finish the proof, we now prove that the double sum in \eqref{eq:1.0} is bounded from above by $n$. Indeed, letting $\gamma=1-\frac{1}{(1+\varepsilon)^{1+1/d}}>0$, we have the following for large enough $n$:
\begin{align*}
    \sum _{k=1}^{n} \sum_{m=1}^{\infty} f_{\delta(k)}(m) &\leq \sum _{k\leq n} \frac{\Gamma\left(1+{1}/{d}\right)}{d\delta(k)^{1+1/d}}+\frac{2}{e\delta(k)}\\
    &\leq \frac{n}{(1+\varepsilon)^{(d+1)/d}{\zeta\left(1+{1}/{d}\right)}}\cdot \sum_{k=1}^{\infty}\frac{1}{k^{1+1/d}}+\frac{(d+1)n^{d/(d+1)}}{(1+\varepsilon)c_d}\cdot \sum_{k=1}^{n}\frac{1}{k}\\
     &\leq \frac{n}{(1+\varepsilon)^{1+1/d}} +\frac{2(d+1)}{(1+\varepsilon) c_d}\cdot n^{d/(d+1)}\log(n)\leq (1-\gamma)n+\gamma n=n\;,
\end{align*}
where the first inequality holds by inequality \eqref{eq:exp_sum_powers}, the second inequality holds as $\frac{1}{\delta(k)^{1+1/d}}=\frac{n}{(k(1+\varepsilon))^{(d+1)/d}}\cdot \frac{d}{\Gamma\left(1+{1}/{d}\right)\zeta\left(1+{1}/{d}\right)}$, and the last inequality holds by the choice of $\gamma$ and provided $n$ is large enough. \qed


\section{The general case: Partitions into positive Powers}\label{sec:proofpowers2}
In this section we prove the upper bound in Theorem \ref{thm:powers} in full generality.
The proof starts similarly to the case of integer $d$, and we will make use of \eqref{eq:exp_sum_powers}, which was proved in the previous section.

Indeed, for every $\varepsilon>0$, we prove by induction on $n$ that $p_d(n) \leq K e^{(1+\varepsilon)c_dn^{1/(d+1)}}$, where $K=K(\varepsilon)$. Note that by choosing $K$ large enough, we can assume that the induction assumption holds for all $n \leq n_0(\varepsilon,d)$, allowing us to assume in what follows that $n \geq n_0(\varepsilon,d)$.
The following holds due to Lemma \ref{lem:recursion}, the induction hypothesis, and by setting $\delta(k)\coloneqq \frac{(1+\varepsilon)c_d k}{(d+1)n^{d/(d+1)}}$:
\begin{align*}
    n\cdot p_{d}(n)&= \sum_{m=1}^{(n+1)^{1/d}-1} \lfloor m^{d}\rfloor\sum _{k=1}^{n/\lfloor m^d\rfloor } p_{d}(n-\lfloor m^d\rfloor k) \leq \sum_{m=1}^{(n+1)^{1/d}-1}m^d \sum _{k=1}^{n/\lfloor m^{d}\rfloor} K e^{(1+\varepsilon)c_d(n-\lfloor m^d\rfloor k)^{1/(d+1)}}\nonumber\\
    &\leq Ke^{(1+\varepsilon)c_d n^{1/(d+1)}}\sum _{k=1}^{n} \sum_{m=1}^{\infty}m^d e^{-\delta(k)\lfloor m^{d}\rfloor}\;,
\end{align*}
where in the second inequality we used the basic fact $(n-x)^{1/(d+1)} \leq n^{1/(d+1)}-\frac{x}{(d+1)n^{d/(d+1)}}$.

To finish the proof, we need to show that the double sum above is at most $n$.
Clearly, the double sum above is at most $A+B$ where
\[
    A=  \sum _{k=1}^{\log^{2}(n)} \sum_{m=1}^{\infty} m^{d}e^{-\delta(k) (m^{d}-1)} \quad\text{and}\quad B=\sum_{k=\log^{2}(n)}^{n}\sum_{m=2}^{\infty} m^{d} e^{-\delta(k) (m^{d}-1)}+\sum_{k=1}^{n}e^{-\delta(k)}\;.
\]
\paragraph{Bounding $A$:}
Let $\gamma>0$ be such that $\left(\frac{1+\gamma \varepsilon}{1+\varepsilon}\right)^{1+1/d}\leq (1-\varepsilon)$. By inequality \eqref{eq:exp_sum_powers} there is a $\delta_0=\delta_0(\gamma, \varepsilon,d)$ so that $ \sum_{m=1}^{\infty} m^{d} e^{-\delta m^{d}} \leq (1+\gamma \varepsilon)d^{-1}\Gamma(1+1/d)/\delta^{1+1/d}$ holds for every $\delta < \delta_0$. Assume $n$ is large enough so that $\delta(k)< \delta_0$ for all $k\leq \log^2(n)$. Then, \begin{align*}\label{eq:after_integral}
     A&\leq  e^{\delta(\log^2(n))}\sum_{k=1}^{\infty} \frac{(1+\gamma \varepsilon)\Gamma(1+1/d)}{d\delta(k)^{1+1/d}}\leq  \frac{(1+\gamma \varepsilon)^{1+1/d}n}{(1+\varepsilon)^{1+1/d}\zeta(1+1/d)}\sum_{k=1}^{\infty}\frac{1}{k^{1+1/d}}\leq  \left(1-\varepsilon\right)n\;,
\end{align*}
The first inequality holds by \eqref{eq:exp_sum_powers} with $\delta=\delta(k)~~(<\delta_0)$, and the second inequality holds by the definition of $\delta(k),c_d$ and as $e^{\delta(\log^2(n))}\leq (1+\gamma \varepsilon)^{1/d}$ provided $n$ is large enough.

\paragraph{Bounding $B$:}
Since $A \leq (1-\varepsilon)n$ it remains to prove that $B \leq \varepsilon n$. Indeed, let $\eta=\eta(d)$ be such that for all $m\geq 2$ we have $m^{d}-1\geq \eta m^{d}$. Then,
\begin{align*}
    B\leq \sum_{k=\log^{2}(n)}^{n}\sum_{m=2}^{\infty} m^{d} e^{-\eta \delta(k) m^{d}} + \frac{1}{1-e^{-\delta(1)}}
    \leq \sum_{k=\log^2(n)}^{\infty} \frac{C_1 \cdot n}{ k^{1+1/d}} +\sum_{k=1}^{n}\frac{C_2 n^{d/(d+1)}}{ k}+\frac{\varepsilon n}{3}
    \leq \varepsilon n\;,
\end{align*}
where the first inequality holds by the definition of $\eta$ and as $\sum_{k=1}^{n}e^{-\delta(k)}$ is a geometric series, the second inequality holds by \eqref{eq:exp_sum_powers} with $\delta = \eta \delta(k)$, where $C_1,C_2$ are constants depending on $d,\varepsilon$, and by assuming $n$ is large enough so that $(1-e^{\delta(1)})^{-1}\leq \varepsilon n /3$; the third inequality holds provided $n$ is large enough so that $\sum_{k=\log^2(n)}^{\infty}1/k^{1+1/d}\leq \frac{\varepsilon}{3C_1}$ and $C_2 n^{d/(d+1)}\log(n)\leq \frac{\varepsilon n}{6}$. \qed

\section{Proof of the upper bound in Theorem \ref{thm:primes}}\label{sec:mainproof}
As before, we will need an evaluation of an `exponential sum'; this time it will run over all primes. We will state this as a Lemma \ref{lem:prime_sum} and postpone its proof for after the proof of the upper bound in Theorem \ref{thm:primes}. We also remark that in this section we only use the upper bound in the following lemma, and the lower bound is used in the following section.

\begin{lemma}\label{lem:prime_sum}
We have $\sum_{p\in \mathcal{P}} p e^{-\delta p} \sim {\delta^{-2}}/{\log(1/\delta)}$ as $\delta \to 0$.
\end{lemma}

To prove the theorem, we will also need the following elementary inequality.

\begin{lemma}\label{lem:Taylor_series}
    Let $f(x)=\sqrt{x/\log(10x)}$. Then, for every $m\geq 1$ and $0\leq x\leq m-1$ we have:
    \[
        f(m-x)\leq f(m)-xf'(m)\;.
    \]
    For all $m>100$ and $0\leq x\leq m/10$ we have
    \[
        f(m-x)\geq f(m)-xf'(m)+x^2f''(m)\;.
    \]
\end{lemma}

\begin{proof}
By Lagrange's remainder theorem, for every $m>1$ and $x$ such that $0< x\leq m$ there is $m-x\leq c\leq m$ such that
\begin{align*}
    f(m-x)&=f(m)- xf'(m)+x^2f''(c)/2\;.
\end{align*}
Noting that $f''(z)=\frac{3-\log^2(10z)}{4z^{3/2}\log^{5/2}(10z)}$ is negative for all $z>e^{\sqrt{3}}/10 \approx 0.565$ we obtain the first part of the lemma for $m>e^{\sqrt{3}}/10$ and $x>0$ with $m-x>e^{\sqrt{3}}/10$.

For the second part of the lemma, for every $m$ and $x\leq m$ we let
\[
    g_m(x)\coloneqq f(m-x)-f(m)+xf(m)-x^2f''(m)\;.
\]
Note that $g_m(0)=g_m'(0)=0$. Hence, proving that for all $m\geq 100$ and $0\leq x\leq m/10$ we have $g''_m(x)>0$ implies the second part of the lemma.
Indeed, let $m\geq 100$ and let $x=\varepsilon m $ with $0\leq \varepsilon\leq 1/10$. We have
\begin{align*}
    g''_m(\varepsilon m)&=\frac{3-\log^2((1-\varepsilon)\cdot 10m)}{4(1-\varepsilon)^{3/2}m^{3/2}\log^{5/2}((1-\varepsilon)\cdot 10m)}-\frac{3-\log^2(10m)}{2m^{3/2}\log^{5/2}(10m)}\\
    &\geq \frac{1}{4m^{3/2}}\left(\frac{1}{(1-\varepsilon)^{3/2}}\cdot \frac{3-\log^{2}(9m)}{\log^{5/2}(9m)}- 2\cdot \frac{3-\log^{2}(10m)}{\log^{5/2}(10m)}\right)\\
    &\geq \frac{1}{4m^{3/2}}\cdot {\frac{2(1-1/10)-(1-\varepsilon)^{-3/2}(1+1/10)}{\sqrt{\log(10m)}}}>0
\end{align*}
where the first inequality holds as $\frac{3-\log^2(10x)}{\log^{5/2}(x)}$ is monotone increasing for $x>e^{\sqrt{15}}$ and as $(1-\varepsilon)10 m\geq 90\geq e^{\sqrt{15}}$, the second inequality holds as $m>100$ which implies both $\frac{3-\log^2(9m)}{\log^{5/2}(9m)}\geq \frac{-(1+1/10)}{\sqrt{\log(10m)}}$ and $\frac{3-\log^2{(10m)}}{\log^{5/2}(10m)}\leq \frac{-(1-1/10)}{\sqrt{\log(10m)}}$, and the last inequality holds as $\varepsilon\leq 1/10$.
\end{proof}

\begin{proof}[Proof of Theorem \ref{thm:primes} (upper bound)]
Set $c=2\pi/\sqrt{3}$ and fix any $0 < \varepsilon<1/2$. We will prove by induction on $n$ that $pp(n) \leq K \cdot e^{(1+\varepsilon)c\sqrt{n/\log(10n)}}$, where $K=K(\varepsilon)$. Note that by choosing $K$ large enough, we can assume that the induction assumption holds for all $n \leq n_0(\varepsilon)$, allowing us to assume in what follows that $n \geq n_0(\varepsilon)$. To simplify the notation, we set\footnote{For $x=0$ we define $f(x)=0$.} $f(x)=\sqrt{x/\log(10x)}$ and denote its derivative by $f'(x)=\frac{1-1/\log(10x)}{2\sqrt{x\log(10x)}}$.
We now have the following inequality:
\begin{align*}
    n\cdot pp(n) &= \sum_{p\in \mathcal{P}\cap[n]} p\sum _{k=1}^{n/p} pp(n-pk) \leq \sum_{p\in \mathcal{P}\cap[n]} p\sum _{k=1}^{n/p} K e^{(1+\varepsilon)cf(n-pk)}\leq Ke^{(1+\varepsilon)cf(n)}\sum _{k=1}^{n} \sum_{p\in \mathcal{P}} pe^{-(1+\varepsilon)ckpf'(n)}\;.
\end{align*}
where the equality holds by Lemma \ref{lem:recursion}, the first inequality holds by the induction hypothesis, the second inequality holds by Lemma \ref{lem:Taylor_series} assuming $n$ is large enough and also by a direct computation when $pk=n$.

To finish the proof, we need to show that the double sum above is at most $n$.
Clearly, for $n$ large enough $(1+\varepsilon)cf'(n)>0$ and thus, the double sum above is at most $A+B$ where
\[
    A=  \sum _{k=1}^{\log^{2}(n)} \sum_{p\in \mathcal{P}} pe^{-{(1+\varepsilon)ckpf'(n)}} \quad\text{and}\quad B=\sum_{k=\log^{2}(n)}^{n}\sum_{a=1}^{\infty} ae^{-kaf'(n)}\;.
\]
\paragraph{Bounding $A$:}
By Lemma \ref{lem:prime_sum} there is a $\delta_0=\delta_0(\varepsilon)$ so that $ \sum_{p\in \mathcal{P}} p e^{-\delta p} \leq (1+\varepsilon/12)\delta^{-2}/\log(1/\delta)$ holds for every $\delta < \delta_0$. Assume $n$ is large enough so that  ${(1+\varepsilon)c\log^{2}(n)f'(n)}< \delta_0(\varepsilon)$. Then, \begin{align*}\label{eq:after_integral}
     A&\leq  \sum_{k=1}^{\log^{2}(n)} \frac{(1+\varepsilon/12)f'(n)^{-2}}{(1+\varepsilon)^2c^2\log\left(\frac{1}{(1+\varepsilon)ckf'(n)}\right)k^2}\leq \frac{8(1+\varepsilon/12)^2n}{(1+\varepsilon)^2c^2}\sum_{k=1}^{\infty} \frac{1}{k^2}\leq  \left(1-\varepsilon\right)n\;.
\end{align*}
The first inequality uses Lemma \ref{lem:prime_sum} with $\delta=(1+\varepsilon)ckf'(n)~~(<\delta_0)$, and the second inequality holds as we have $\frac{f'(n)^{-2}}{-\log\left({(1+\varepsilon)ckf'(n)}\right)}\leq 8(1+\varepsilon/12)n$ for all $1\leq k\leq \log^{2}(n)$, and large $n$.

\paragraph{Bounding $B$:}
Since $A \leq (1-\varepsilon)n$, it remains to prove that $B \leq \varepsilon n$. Indeed,
\begin{align*}
    B = \sum_{k=\log^{2}(n)}^{n}\frac{e^{-kf'(n)}}{\left(1-e^{-kf'(n)}\right)^2}
    \leq \frac{4n\log(10n)}{(1-1/\log(10n))^{2}}\sum _{k=\log^{2}(n)}^{n} \frac{1}{k^{2}}\leq \frac{5n\log(10n)}{\log^2(n)-1}\leq \varepsilon n\;.
\end{align*}
The equality holds as $\sum_{k=1}^{\infty}kt^k=\frac{t}{(1-t)^2}$, the first inequality uses the elementary
fact 
(see Lemma \ref{lem:estimations}) $\frac{e^{-x}}{(1-e^{-x})^2}\leq \frac{1}{x^2}$, and the second inequality holds as for all $t>1$ we have $\sum_{s=t}^{\infty}\frac{1}{s^2}\leq \int_{t-1}^{\infty}\frac{1}{x^2}dx=\frac{1}{t-1}$.
\end{proof}

\begin{proof}[Proof of Lemma \ref{lem:prime_sum}]
We start by showing the upper bound, that is, for every $\varepsilon>0$ there exists $\delta_0>0$ such that for all $\delta<\delta_0$ we have
\[
    \sum_{p\in\mathcal{P}}pe^{-p\delta}\leq \frac{1+\varepsilon}{\delta^2\log(1/\delta)}\;.
\]

First, we show that for every $\xi>0$ there exists $\eta_0>0$ such that for all $\eta<\eta_0$ the following holds:
\begin{equation}\label{eq:evaluating_the_integral_primes_upper}
        \int_{1}^{\infty}x\log(x)e^{-\eta x\log(x)}dx\leq \frac{1+\xi}{\eta^2\log(1/\eta)}\;.
\end{equation}
Indeed, let $0<\xi<1$ and let $a>0$ be sufficiently large so that $\log(x)+1\geq (1-\xi/8)\log(x\log(x))$ for all $x\geq a$. Then,
\begin{align*}\label{eq:evaluation_with_integral}
    \int_{1}^{\infty}x\log(x)e^{-\eta x\log(x)}dx &\leq \int_{1}^{a}x\log(x)e^{-\eta x\log(x)}dx+\frac{1}{(1-\xi/8)\eta^2}\int_{\eta a\log(a)}^{\infty}\frac{z e^{-z}}{\log(z/\eta)}dz\nonumber\\
    &\leq \int_{1}^{a}x^2dx+\frac{1+\xi/4}{\eta^2}\left(\int_{\eta a\log(a)}^{1/\log^2(1/\eta)}\frac{ze^{-z}}{\log(z/\eta)}dz+\int_{1/\log^2(1/\eta)}^{\infty}\frac{ze^{-z}}{\log(z/\eta)}dz\right)\nonumber\\
    &\leq \frac{\xi/2}{\eta^2\log(1/\eta)}+\frac{1+\xi/4}{\eta^2}\left(\frac{C}{\log^2(1/\eta)}+\frac{1}{\log(1/\eta)-2\log\log(1/\eta)}\int_{0}^{\infty}ze^{-z}dz\right)\\
    &\leq \frac{1+\xi}{\eta^2\log(1/\eta)}\;,
\end{align*}
where the first inequality holds by the substitution $z=\eta x \log x$, which, by the choice of $a$, guarantees that
$dz=\eta(\log(x)+1)dx\geq \eta(1-\xi/8)\log(z/\eta)dx$, the third inequality holds for some constant $C$ that depends only on $a$, and provided $\eta_0$ is small enough, the last inequality holds provided $\eta$ is sufficiently small and the fact that $\int_{0}^{\infty}ze^{-z}dz=1$.

Next, for every $\delta,x>0$ let $f_\delta(x)=x\log(x)e^{-\delta x\log(x)}$. We now show that for any $\xi>0$ there exists $\eta_0>0$ so that for all $\eta<\eta_0$ we have
\begin{equation}\label{eq:sum_to_int}
    \sum_{m=2}^{n} f_\eta(m)\leq \int_{1}^{\infty}f_\eta(x) dx+\frac{\xi}{\eta^2\log(1/\eta)}\;.
\end{equation}
Choosing the parameters appropriately, the prime number theorem, \eqref{eq:evaluating_the_integral_primes_upper}, and \eqref{eq:sum_to_int} assert that
\[
    \sum_{p\in \mathcal{P}}pe^{-\delta p}\sim \sum_{m=2}^{\infty}f_\delta(m)\leq \int_{1}^{\infty}f_{\delta}(x)dx+\frac{\varepsilon/2}{\delta^2\log(1/\delta)}\leq \frac{1+\varepsilon}{\delta^2\log(1/\delta)} \;.
\]
which finishes the proof of the upper bound. To prove \eqref{eq:sum_to_int}, note that for any fixed $\eta$ we have
$\frac{d}{dx}f_\eta(x)=-(\log(x)+1)(\eta x\log(x)-1)e^{-\eta x\log(x)}$.
Hence, provided $\eta$ is small enough and letting $b\sim \eta^{-1}/\log(1/\eta)>1/e$ be such that $f_\eta(x)$ is maximal, we have the following as $f_\eta(x)$ is increasing from $1/e$ to $b$, and decreasing from $b$ to infinity:
\[
    \sum_{m=2}^{\lfloor b \rfloor -1}f_\eta(m)\leq \int_{1}^{\lfloor b\rfloor}f_\eta(x)dx\quad \text{and}\quad \sum_{m=\lfloor b \rfloor+2}^{\infty} f_\eta(m)\leq \int_{\lfloor b\rfloor}^{\infty}f_\eta(x) dx\;.
\]

We see that the integral $\int_{0}^{\infty}f_\eta(x)dx$ accounts for $\sum_{m=1}^{\infty}f_\eta(m)$ except possibly for $f(\lfloor b\rfloor)$ and $f(\lfloor b\rfloor +1)$, hence, we obtain \eqref{eq:sum_to_int} as follows
\begin{align*}
    \sum_{m=2}^{\infty}f_\eta(m)&\leq \int_{1}^{\infty}f_\eta(x)dx+2f_\eta(b)=\int_{1}^{\infty}f_\eta(x)dx+\frac{2}{\eta e}\leq \int_{1}^{\infty}f_\eta(x)dx+\frac{\xi}{\eta^2\log(1/\eta)}\;,
\end{align*}
where the first inequality holds by the above inequalities and the definition of $b$ being the maximal point of $f_\eta(x)$, and the second inequality holds provided $\eta$ is sufficiently small.

The proof of the lower bound is almost identical, therefore, we will only sketch it. One can obtain analogous inequalities to \eqref{eq:evaluating_the_integral_primes_upper} and \eqref{eq:sum_to_int}, that is
\[
    \int_1^{\infty}f_\eta(x)dx\geq \frac{1-\xi}{\eta^2\log(1/\eta)}\quad \text{and}\quad
    \sum_{m=2}^{n} f_\eta(m)\geq \int_{1}^{\infty}f_\eta(x)dx-\frac{\xi}{\eta^2\log(1/\eta)}\;.
\]
Using the prime number theorem and the above inequalities, we derive the lower bound, similar to the upper bound.
\end{proof}

\section{Proof of the lower bound in Theorem \ref{thm:primes}}\label{sec:proofprimelower}
Set $c=2\pi/\sqrt{3}$ and fix any $0 < \varepsilon<1/6$. We will prove by induction on $n$ that $pp(n) \geq \frac{1}{K}e^{(1-\varepsilon)c\sqrt{n/\log(10n)}}$, where $K=K(\varepsilon)$. Note that by choosing $K$ large enough, we can assume that the induction assumption holds for all $n \leq n_0(\varepsilon)$, allowing us to assume in what follows that $n \geq n_0(\varepsilon)$. To simplify the notation, we set\footnote{For $x=0$ we define $f(x)=0$.} $f(x)=\sqrt{x/\log(10x)}$, denote its derivative by $f'(x)=\frac{1-1/\log(10x)}{2\sqrt{x\log(10x)}}$, and its second derivative by $f''(x)=\frac{-1+3/\log^2(10x)}{4x\sqrt{x\log(10x)}}$. We now have
\begin{align}\label{eq:less_than_n_primes}
    n\cdot pp(n) &= \sum_{p\in \mathcal{P}\cap[n]} p\sum _{k=1}^{n/p} pp(n-pk) \geq \sum_{p\in \mathcal{P}\cap[n]} p\sum _{k=1}^{n/p} \frac{1}{K} e^{(1-\varepsilon)cf(n-pk)}\nonumber\\
    &\geq \frac{1}{K}e^{(1-\varepsilon)cf(n)}\sum _{k=1}^{\log(n)} \sum_{\mathcal{P}\ni p\leq n^{3/4}} pe^{-(1-\varepsilon)ckpf'(n)}\cdot e^{2ck^2p^2f''(n)}\nonumber \\
    &\geq \frac{1}{K}e^{(1-\varepsilon)cf(n)}\sum _{k=1}^{\log(n)} \sum_{\mathcal{P}\ni p\leq n^{3/4}} pe^{-(1-\varepsilon)ckpf'(n)}(1+2ck^2p^2f''(n))\;,
\end{align}
where the equality holds by Lemma \ref{lem:recursion}, the first inequality holds by the induction hypothesis, the second inequality holds by Lemma \ref{lem:Taylor_series} with
$x=pk$ while noting that for large enough $n$ we have $pk\leq n^{3/4}\log(n) \leq n/10$, and the last inequality holds as for all $x$ we have $1+x\leq e^{x}$.

Hence, to complete the proof it remains to establish that the double sum in \eqref{eq:less_than_n_primes} is bounded from below by $n$.
The double sum in \eqref{eq:less_than_n_primes} is at least $A-B+C$ where (note that $f''(n) < 0 < f'(n)$)
\[
    A= \sum _{k=1}^{\log(n)}\sum_{p\in \mathcal{P}} pe^{-(1-\varepsilon) ckp f'(n)} \quad\text{,}\quad B=\sum_{k=1}^{\log(n)}\sum_{m=n^{3/4}}^{\infty} me^{- ckm f'(n)/2}\;,
\]
\[
    C=\sum _{k=1}^{\log(n)}\sum_{m=1}^{\infty}  2c k^2 m^3f''(n)e^{- 5ckm f'(n)/6}\;.
\]

\paragraph{Bounding $A$:}
By Lemma \ref{lem:prime_sum} there is a $\delta_0=\delta_0(\varepsilon)$ so that $\sum_{p\in \mathcal{P}}pe^{-\delta p} \geq  {(1-\varepsilon/4)\delta^{-2}}/{\log(1/\delta)}$ holds for every $\delta < \delta_0$.
Assume $n$ is large enough so that $(1-\varepsilon/2)c  \log(n) f'(n)\leq \delta_0$. Then,
\begin{align*}
     A&\geq  \sum_{k=1}^{\log(n)} \frac{(1-\frac{\varepsilon}{4})f'(n)^{-2}}{(1-\frac{\varepsilon}{2})^2c^2 \log\left(\frac{1}{(1-\frac{\varepsilon}{2})c k f'(n)}\right)k^2} \geq \frac{8(1-\frac{\varepsilon}{4})^{3/2}n}{(1-\frac{\varepsilon}{2})^2c^2}\sum_{k=1}^{\log(n)} \frac{1}{k^2} \geq \left(1+\frac{\varepsilon}{2}\right)n\;,
\end{align*}
where the first inequality holds by Lemma \ref{lem:prime_sum} applied with $\delta=(1-\varepsilon/2)c k f'(m) ~~(\leq \delta_0)$, the second inequality holds as $\frac{f'(n)^{-2}}{-\log\left({(1-\varepsilon/2)c k f'(n)}\right)}\geq 8\sqrt{1-\varepsilon/4} \cdot n$ for all $1\leq k\leq \log(n)$ and large $n$. The last inequality holds provided $n$ is large enough so that $\sum_{k=1}^{\log(n)}\frac{1}{k^2}\geq \sqrt{1-\varepsilon/4}\cdot \pi^2/6$.

\paragraph{Bounding $B$:}
Observe that the following holds for large enough $n$:
\begin{align*}
    B&\leq 2\log(n)n^{3/4}\cdot \frac{e^{- c f'(n) n^{3/4}/2}}{\left(1-e^{-cf'(n)/2}\right)^2}\leq  \frac{ 2\log(n) n^{3/4} \cdot e^{-cf'(n) n^{3/4}/2}}{c^2\left(f'(n)-f'(n)^2\right)^2/4} \leq \frac{1}{4}\varepsilon n\;,
\end{align*}
where the first inequality holds as for all $M\geq 1$ and $0<x<1$ we have\footnote{Note that $\sum_{k=M}^{\infty} k x^{k} = x^{M}\sum_{k=0}^{\infty}(k+M)x^{k} =x^{M}\sum_{k=0}^{\infty}kx^{k} + Mx^{M}\sum_{k=0}^{\infty}x^{k} = \frac{Mx^{M}}{(1-x)^{2}}+\frac{Mx^{M}}{1-x}  \leq \frac{(M+1)x^M}{(1-x)^2}$.} $\sum_{k=M}^{\infty} k x^{k}\leq \frac{(M+1)x^{M}}{(1-x)^2}$, the second inequality holds as for all $0<x<1$ we have $(1-e^{-x})^2>(x-x^2)^2$.

\paragraph{Bounding $C$:}
Since $A-B\geq (1+\varepsilon/4)n$, it is enough to prove that $C\geq -\frac{1}{4}\varepsilon n$. Indeed,
\begin{align*}
    C&\geq cf''(n) \sum _{k=1}^{\log(n)}k^2\frac{12e^{-5c k f'(n)/6}}{\left(1-e^{-5c k f'(n)/6}\right)^4} \geq\frac{6^6 f''(n)}{5^4c^3f'(n)^4}  \sum_{k=1}^{\infty} \frac{1}{k^2}\geq -\frac{1}{4}\varepsilon n\;,
\end{align*}
where the first and second inequalities hold as for all $0<x< 1$ we have\footnote{This follows by taking the derivative of the identity $\sum_{k=0}^{\infty}x^{k}=\frac{1}{1-x}$ three times with respect to $x$.} $\sum_{k=0}^{\infty}k^3x^{k}\leq \frac{6x}{(1-x)^4}$ and\footnote{See Lemma \ref{lem:estimations} in the appendix for a short proof.} $e^{-x}/(1-e^{-x})^4\leq 3/x^4$ for $0<x<1$. The last inequality holds for large enough $n$. \qed

\appendix
\section{Some Elementary inequalities}

\begin{lemma}\label{lem:estimations}
For all $x>0$ we have
\begin{equation}\label{eq1}
    \frac{e^x}{(e^x-1)^2}\leq \frac{1}{x^2}\;,
\end{equation}
and
\begin{equation}\label{eq2}
    \frac{e^{-x}(e^{-x}+1)}{(1-e^{-x})^3}\leq \frac{2}{x^3}\;,
\end{equation}
and for all $0<x<1$ we have
\begin{equation}\label{eq4}
        \frac{e^{-x}}{(1-e^{-x})^4}=\frac{e^{3x}}{(e^x-1)^4}\leq \frac{3}{x^4}\;.
    \end{equation}
\end{lemma}
\begin{proof}
The proofs of all inequalities are similar; we expand $e^x$ into its power series, and then compare the coefficients on both
sides of the inequality. Inequality \eqref{eq1} is equivalent to
$
    x^2e^x\leq e^{2x}-2e^{x}+1.
$
By expanding each side to its power series, it suffices to show that
\[
    \sum_{n=2}^{\infty}\frac{x^n}{(n-2)!} \leq \sum_{n=2}^{\infty}\frac{2^n-2}{n!}x^n\;.
\]
To see this, note that the coefficient of $x^n$ on the left-hand side is always at most the coefficient of $x^n$ on the right-hand side for all $n\geq 2$ as clearly $n(n-1)\leq 2^n-2$ for all $n\geq 2$.

To prove \eqref{eq2} note that it is equivalent to
\[
    x^3(e^x+e^{2x})\leq 2e^{3x}-6e^{2x}+6e^{x}-2\;.
\]
Replacing $e^x$ with its power series expansion, this is equivalent to
\[
    \sum_{n=3}^{\infty}\frac{1+2^{n-3}}{(n-3)!}x^{n}\leq
    \sum_{n=3}^{\infty}\frac{2\cdot 3^n-6\cdot 2^n+6}{n!}x^{n}\;.
\]
To establish this inequality for all $x>0$, we claim that the coefficient of $x^n$ in the right-hand series is at least as large as the coefficient of $x^n$ in the left-hand side. To prove this, we will show that for all $n \geq 3$ we have
\[
    n(n-1)(n-2)(1+2^n/8)+6\cdot 2^n\leq 2\cdot 3^n\;.
\]
First, for $n\geq 20$ we indeed have the above inequality as
$$
    n(n-1)(n-2)(1+2^{n-3})+6\cdot 2^n \leq  \binom{n}{5}{2^{n-4}}+\binom{n}{2}2^{n-1}+ n2^n+2^{n+1} \leq \sum_{k=0}^n \binom{n}{n-k}2^{k+1}= 2\cdot 3^n\;,
$$
where in the first inequality we used $6\cdot 2^{n}\leq n2^n+2^{n+1}$, which holds for all $n\geq 4$, and $n(n-1)(n-2)\leq \binom{n}{2}2^{n-1}$ which holds for all $n$, and also $ n(n-1)(n-2)\frac{2^n}{8}\leq \binom{n}{5}\frac{2^{n}}{16}$ which holds for all $n\geq 20$.
The first two inequalities we mentioned above are straightforward to prove, the third is equivalent to $(n-3)(n-4)\geq 240$, which holds for all $n\geq 20$ as then we have $(n-3)(n-4)\geq (n-4)^2\geq 16^2>240$.
For $n\leq 19$, one can verify that the coefficient of $x^n$ in the right-hand series is at least the coefficient of $x^n$ in the left-hand side series.

Lastly, we prove \eqref{eq4}. Clearly, it is enough to prove that $\frac{e^{2x}}{(e^{x}-1)^2}\leq \frac{1}{x^2}+\frac{1}{x}+1$, as then using inequality \eqref{eq1}, we are done.
The above is equivalent to proving that
    $
        2x^2e^{x}\leq (1+x)(e^{2x}-2e^{x}+1)+x^2.
    $
    By expanding each side to its power series, we wish to show that
    \[
        \sum_{n=4}^{\infty}\frac{2n(n-1)}{n!}x^n\leq \sum_{n=4}^{\infty} \frac{2^n-2+n(2^{n-1}-2)}{n!}x^{n}\;.
    \]
    As in previous cases, this holds since the coefficient of $x^{n}$ on the left-hand side is at most as large as the coefficient of $x^{n}$ on the right-hand side.
\end{proof}

\end{document}